\documentclass[reqno]{amsart}

\usepackage{graphicx}
\usepackage{amscd}
\usepackage{amsmath, amssymb}
\usepackage{graphics}
\usepackage{graphicx}
\usepackage{epsfig}
\usepackage{xcolor}
\usepackage{soul}
\usepackage{tikz}

\newtheorem{theorem}{Theorem}

\newtheorem{corollary}[theorem]{Corollary}

\newtheorem{definition}[theorem]{Definition}

\newtheorem{lemma}[theorem]{Lemma}

\newtheorem{proposition}[theorem]{Proposition}
\newtheorem{remark}[theorem]{Remark}

\newtheorem{question}[theorem]{Question}

\begin{document}

\title[]{On representations of the triplet group and some of its extensions}

\author{Mohamad N. Nasser}

\address{Mohamad N. Nasser\\
         Department of Mathematics and Computer Science\\
         Beirut Arab University\\
         P.O. Box 11-5020, Beirut, Lebanon}
         
\email{m.nasser@bau.edu.lb}
\author{Nafaa Chbili}

\address{Nafaa Chbili\\
         Department of Mathematical Sciences\\
         United Arab Emirates University\\
         15551, Al Ain, UAE}
         
\email{nafaachbili@uaeu.ac.ae}

\author{Khaled Qazaqzeh}

\address{Khaled Qazaqzeh\\
         Department of Mathematics\\
         Yarmouk University\\
         Irbid, Jordan}
   
\email{qazaqzeh@yu.edu.jo}
\maketitle
 
\begin{abstract}
In this paper, we study the representations of the triplet group $L_n$, where $n$ is a positive integer, together with their extensions to the virtual and welded triplet groups $VL_n$ and $WL_n$, respectively. We first introduce $L_n$, its extensions, and its pure subgroup. We then investigate several representations, proving the irreducibility of the classical Tits representation $\Theta: L_n \longrightarrow\mathrm{GL}_{n-1}(\mathbb{C})$ over the complex field $\mathbb{C}$ and constructing a new representation $\mu: L_n \longrightarrow \mathrm{GL}_{n}(\mathbb{Z}[t^{\pm 1}])$, where $t$ is an indeterminate. For the representation $\mu$, we completely study its faithfulness and irreducibility. We also classify all complex homogeneous $2$-local representations of $L_n$ for $n \ge 3$ and all complex non-homogeneous $2$-local representations of $L_3$, establishing connections with the complex specialization of the representation $\mu$. Finally, we examine extensions of $L_n$ representations to $VL_n$ and $WL_n$, proving their existence, classifying non-trivial complex homogeneous $2$-local representations, and analyzing their faithfulness and irreducibility. The paper concludes with an open question concerning further extensions of representations of $L_n$ to $VL_n$ and $WL_n$.
\end{abstract}

\renewcommand{\thefootnote}{}
\footnote{\textit{Keywords and Phrases.} Triplet Group; Virtual Triplet Group; Welded Triplet Group; Irreducible Representations; Faithful Representations.}
\footnote{\textit{Mathematics Subject Classification.} 20F36.}

\vspace*{0.2cm}

\section{Introduction}
Braid-like groups occupy a central position in the interaction between algebra, geometry, and topology, with strong ties to knot theory, configuration spaces, mapping class groups, and representation theory. Since Artin’s introduction of the braid group, many generalizations have been developed, driven by both geometric motivations and the need to encode new equivalence relations on braids and links. Consequently, the study of their algebraic structure and mutual relationships has become an active area of research.
\vspace{0.1cm}

The braid group $B_n$ is an abstract group generated by $\sigma_1,\sigma_2,\ldots,\sigma_{n-1}$ \cite{EArtin1926,EArtin1947}. The study of linear representations of $B_n$ and its extensions, especially questions concerning faithfulness and irreducibility, has been the subject of extensive study. Recall that a group is called linear if it admits a faithful representation into a general linear group. For many years, it was unclear whether the braid group $B_n$ admits a faithful representation. An early attempt used the Burau representation \cite{WBurau1936}, which is faithful for $n \le 3$ \cite{JBirman1974} but was later shown to be unfaithful for $n \ge 5$ \cite{Moo, Dlong1992, SBigelow1999}, while the case $n = 4$ remains open. This problem was eventually resolved by the Lawrence-Krammer-Bigelow representation \cite{RLaw1990}, which was proved to be faithful for all $n \ge 2$ \cite{SBigelow2000, DKram2002}. Hence, the braid group $B_n$ is linear. The most important subgroup of $B_n$ is the pure braid group, denoted by $P_n$, which is the normal subgroup defined as the kernel of the natural homomorphism $B_n \longrightarrow S_n$, where $S_n$ is the symmetric group on $n$ elements, given by $\sigma_i \longmapsto (i\ \ i+1)$ for $1\leq i\leq n-1$. Regarding extensions of $B_n$, L. Kauffman introduced the virtual braid group, denoted by $VB_n$, which is an important group in the study of virtual knot theory \cite{LKauff1999, LKauff2004}.

\vspace{0.1cm}

Coxeter groups form another fundamental class of groups in algebra and geometric group theory. A Coxeter group $C$ with generators $c_1,c_2,\ldots,c_r$ is defined by the presentation
\[
C=\langle c_1,c_2,\ldots,c_r \mid c_i^2=1,\ (c_ic_j)^{m_{ij}}=1,\ 1\leq i,j\leq r\rangle,
\]
where $m_{ij}=m_{ji}\geq 2$ for $i\neq j$. These relations reflect both the combinatorial and geometric structure of the group. A classical example is the symmetric group on $n$ elements $S_n$, which admits a Coxeter presentation with generators $\alpha_1,\alpha_2,\ldots,\alpha_{n-1}$ and relations
\begin{enumerate}
\item $\alpha_i^2=1$ for $1\leq i\leq n-1$,
\item $\alpha_i\alpha_{i+1}\alpha_i=\alpha_{i+1}\alpha_i\alpha_{i+1}$ for $1\leq i\leq n-2$,
\item $\alpha_i\alpha_j=\alpha_j\alpha_i$ for $|i-j|\geq 2$.
\end{enumerate}
\vspace{0.1cm}

Another Coxeter group of interest is the triplet group, denoted by $L_n$, which is generated by the elements $\ell_1,\ell_2,\ldots,\ell_{n-1}$ in Figure \ref{tripletgenerator}, subject to the relations in Figures  \ref{tripletrelation1} and  \ref{tripletrelation2}. This group 
was introduced by M.~Khovanov in his study of $K(\pi,1)$ subspace arrangements \cite{MKhovanov1996}. The triplet group admits a geometric interpretation through certain topological objects called noodles. Fix a codimension-one foliation on the $2$-sphere, allowing possibly for singular points. A noodle is defined as a finite collection of closed curves on the sphere satisfying the following conditions: no two intersection points lie on the same leaf of the foliation, no four curves intersect at a single point, and no intersection point occurs at a singular point of the foliation. The pure triplet group, denoted by $PL_n$, is defined analogously to the pure braid group, as the kernel of the natural homomorphism $L_n \longrightarrow S_n$ given by $\ell_i \longmapsto (i\ \ i+1)$ for $1\leq i\leq n-1$ \cite{TNaik2021}. It was shown in the same reference that, for $n\geq 4$, $PL_n$ is a non-abelian free group of finite rank. In \cite{PKumar2024}, P. Kumar, T. Naik, N. Nanda, and M. Singh introduced the virtual triplet group, denoted by $VL_n$, defined in analogy with the virtual braid group, and provided several presentations of this group. \vspace{0.1cm}

Representations of the triplet group $L_n$, as well as of its extensions and subgroups, play a central role in understanding its algebraic structure. Among these, the Tits representation \cite{Jtits1961} is one of the most fundamental; it was proved faithful by J.~Tits, thereby establishing the linearity of all Coxeter groups. Despite this, relatively few representations of $L_n$ have been studied, and little attention has been given to their generalizations or to the investigation of its subgroups via representation theory. Motivated by this idea, the aim of this paper is to investigate the algebraic structure of $L_n$ through the study of its extensions and subgroups, and to construct and analyze a broader family of representations together with their properties. \vspace{0.1cm}

The paper is structured in three main sections. The first section introduces the triplet group $L_n$, its extensions, and its subgroups. In Section 2, we explore various representations of $L_n$ and their properties. We first prove the irreducibility of the Tits representation $\Theta: L_n \longrightarrow \mathrm{GL}_{n-1}(\mathbb{C}) $
(Theorem \ref{Tits irr}), then we construct a new representation $
\mu : L_n \longrightarrow \mathrm{GL}_{n}(\mathbb{Z}[t^{\pm 1}]) $ (Theorem \ref{defmuu}) studying its faithfulness and irreducibility (Theorems \ref{Thmfa} and \ref{Thmir}). We also classify all complex homogeneous $2$-local representations of $L_n$ for $n\ge 3$ (Theorem \ref{TH13}) and relate them to the complex specialization of the representation $\mu$, along with a classification of all complex non-homogeneous $2$-local representations of $L_3$ (Theorem \ref{Thhh222}).  The final section examines the extension of $L_n$ representations to the virtual and welded triplet groups, $VL_n$ and $WL_n$, demonstrating their existence (Proposition \ref{prop11}), classifying all non-trivial complex homogeneous $2$-local representations for $n \ge 3$ (Theorem \ref{Them111}), and investigating their faithfulness and irreducibility (Theorems \ref{Thmfa111} and \ref{irrtheee}). The paper concludes with an open question on further extensions.

\vspace*{0.1cm}

\section{The Triplet Group and Some of Its Extensions}

We begin by introducing the fundamental definitions and presentations of the braid group, the pure braid group, the virtual braid group, and the welded braid group as the subsequent constructions will be formulated in direct analogy with these groups.

\begin{definition} \cite{EArtin1926, EArtin1947}
The braid group on $n$ strands, denoted by $B_n$, is the group generated by the elements
\[
\sigma_1, \sigma_2, \ldots, \sigma_{n-1},
\]
subject to the relations
\begin{align}
\sigma_{i+1}\sigma_i\sigma_{i+1}
    &= \sigma_i\sigma_{i+1}\sigma_i,
    && i=1,2,\ldots,n-2, \label{eqs1} \\
\sigma_i\sigma_j
    &= \sigma_j\sigma_i,
    && |i-j|\ge 2. \label{eqs2}
\end{align}
\end{definition}

The generator $\sigma_i$ and its inverse are often visualized as in
Figure~\ref{braidgenerator}, while the relation
$\sigma_{i+1}\sigma_i\sigma_{i+1}=\sigma_i\sigma_{i+1}\sigma_i$
is illustrated in Figure~\ref{braidrelation}.


\begin{figure}[h]
	\centering
\includegraphics[width=4.5cm,height=2cm]{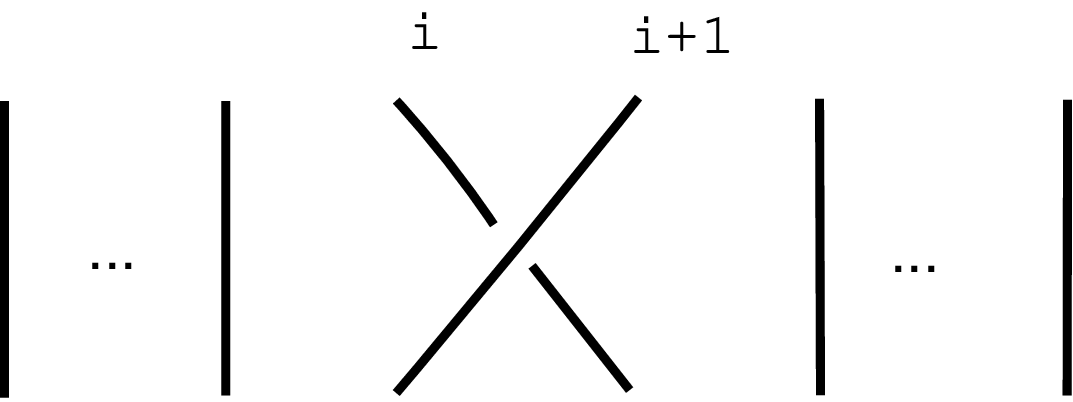}\hspace{2cm} \includegraphics[width=4.5cm,height=2cm]{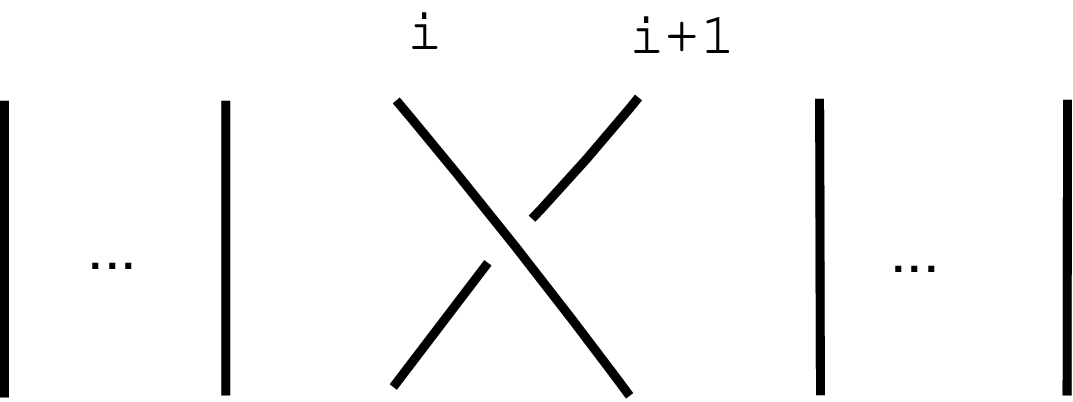}
\caption{The braid group generator $\sigma_i$ and its inverse  $\sigma_i^{-1}$.}
\label{braidgenerator}
\end{figure}

\begin{figure}[h]
	\centering
\includegraphics[width=10cm,height=2.2cm]{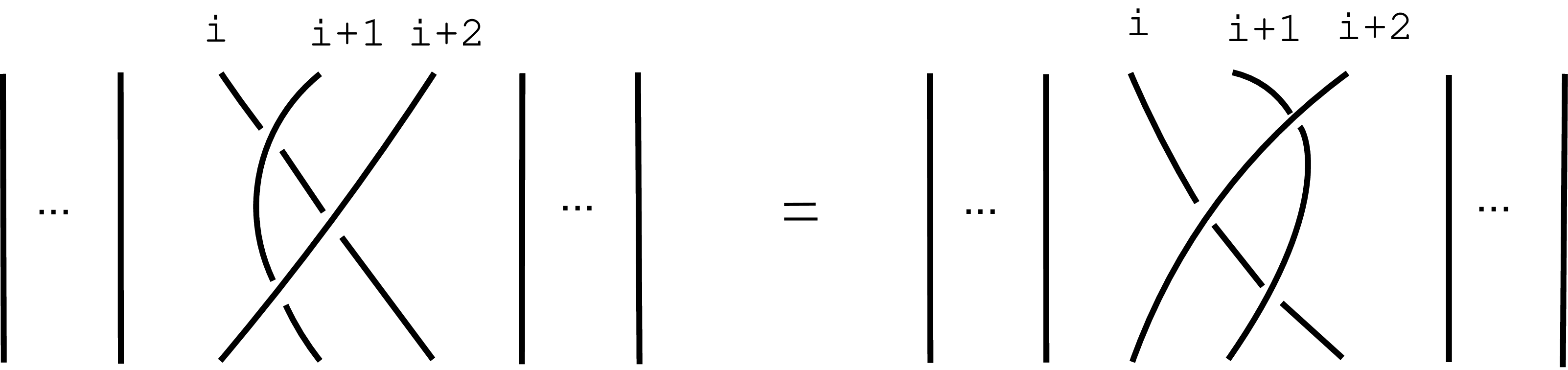}
	\caption{The relation  $\sigma_{i}\sigma_{i+1}\sigma_{i}=\sigma_{+1}\sigma_{i}\sigma_{i+1}$.}
\label{braidrelation}
\end{figure}

\begin{definition} \cite{EArtin1926, EArtin1947}
The pure braid group on $n$ strands, denoted by $P_n$, is defined as the kernel of the homomorphism $B_n \longrightarrow S_n$ defined by $\sigma_i \longmapsto (i \ \ i+1)$, $1\leq i \leq n-1$, where $S_n$ is the symmetric group on $n$ elements. It admits a presentation with the following generators.
$$A_{ij}=\sigma_{j-1}\sigma_{j-2} \ldots \sigma_{i+1}\sigma^2_{i}\sigma^{-1}_{i+1} \ldots \sigma^{-1}_{j-2}\sigma^{-1}_{j-1}, \qquad 1\leq i<j\leq n.$$
\end{definition}

\begin{definition}\cite{LKauff1999}
The virtual braid group on $n$ strands, denoted by $VB_n$, is the group generated by the elements
\[
\sigma_1, \sigma_2, \ldots, \sigma_{n-1}
\]
of the braid group $B_n$ together with another family of elements, namely
$$\rho_1,\rho_2, \ldots, \rho_{n-1}.$$ 
In addition to the relations \eqref{eqs1} and \eqref{eqs2} that define $B_n$, the generators of $VB_n$ satisfy the following relations.
\begin{align}
\rho_i^2
    &= 1,
    && i=1,2,\ldots,n-1, \label{eqs11v} \\
\rho_i\rho_j
    &= \rho_j\rho_i,
    && |i-j|\ge 2, \label{eqs22v} \\
\rho_i\rho_{i+1}\rho_i
    &= \rho_{i+1}\rho_i\rho_{i+1},
    && i=1,2,\ldots,n-2, \label{eqs33v} \\
\sigma_i\rho_j
    &= \rho_j\sigma_i,
    && |i-j|\ge 2, \label{eqs44v} \\
\rho_i\rho_{i+1}\sigma_i
    &= \sigma_{i+1}\rho_i\rho_{i+1},
    && i=1,2,\ldots,n-2. \label{eqs55v}
\end{align}
\end{definition}


\begin{definition} \cite{RFenn1997}
The welded braid group on $n$ strands, denoted by $WB_n$, is the group defined as the quotient of $VB_n$ by adding the following relation.
\begin{equation} \label{eqs1W}
\rho_i\sigma_{i+1}\sigma_i=\sigma_{i+1}\sigma_i\rho_{i+1}, \qquad i=1,2,\ldots,n-2.
\end{equation}
\end{definition}

We now proceed to give the definitions of the triplet group, the pure triplet group, the virtual triplet group, and the welded triplet group.

\begin{definition} \cite{MKhovanov1996}
The triplet group on $n$ strands, denoted by $L_n$, is the group generated by the elements
\[
\ell_1, \ell_2, \ldots, \ell_{n-1},
\]
subject to the relations
\begin{align}
\ell_i^2
    &= 1,
    && i=1,2,\ldots,n-1, \label{eq:L1} \\
\ell_i\ell_{i+1}\ell_i
    &= \ell_{i+1}\ell_i\ell_{i+1},
    && i=1,2,\ldots,n-2. \label{eq:L2}
\end{align}
\end{definition}
\noindent We can see that\\
$\bullet$ $L_2=\langle \ell_1 \ | \ \ell_1^2=1 \rangle=\{e,\ell_1 \}\cong \mathbb{Z}_2.$\\
$\bullet$ $L_3=\langle \ell_1,\ell_2 \ | \ \ell_1^2=\ell_2^2=1, \ell_1\ell_2\ell_1=\ell_2\ell_1\ell_2\rangle=\{e,\ell_1, \ell_2, \ell_1\ell_2, \ell_2\ell_1, \ell_1\ell_2\ell_1\}\cong S_3.$

\begin{figure}[h]
\centering
		\includegraphics[width=5cm,height=2cm]{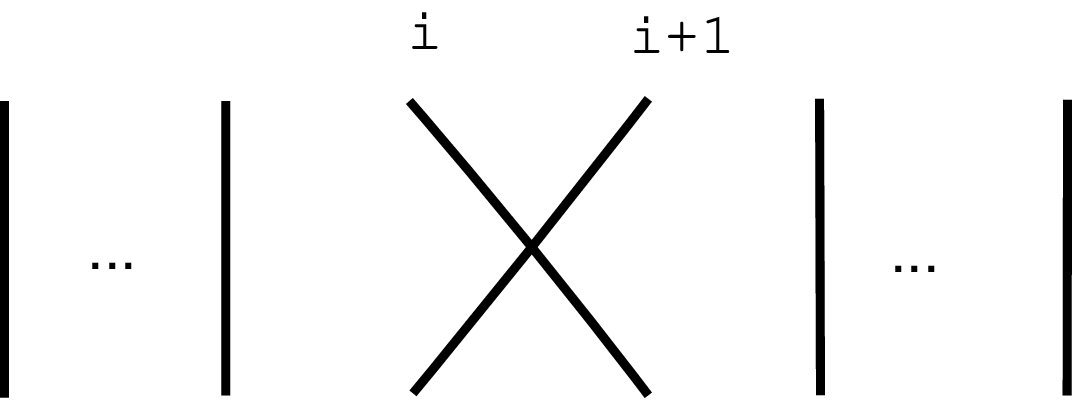}
	\caption{The generator $\ell_i.$}
\label{tripletgenerator}

\end{figure}

\begin{figure}[h]
	\centering
\includegraphics[width=10cm,height=2.2cm]{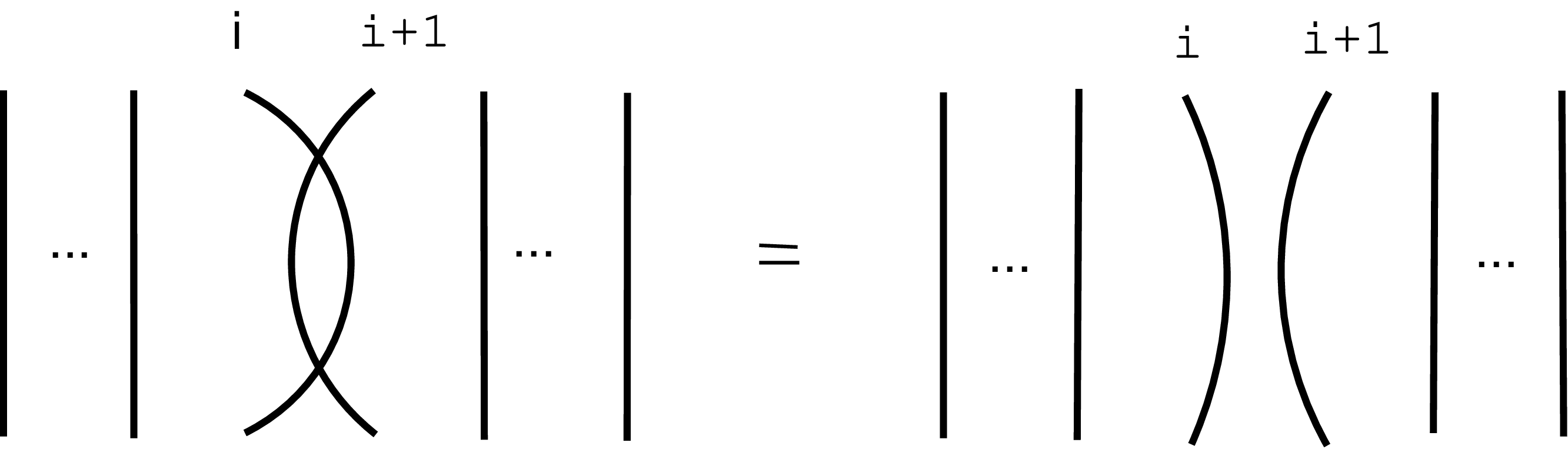}
	\caption{The relation  $\ell_i^2 =1$.}
\label{tripletrelation1}
\end{figure}

\begin{figure}[h]
\centering
\includegraphics[width=10cm,height=2.2cm]{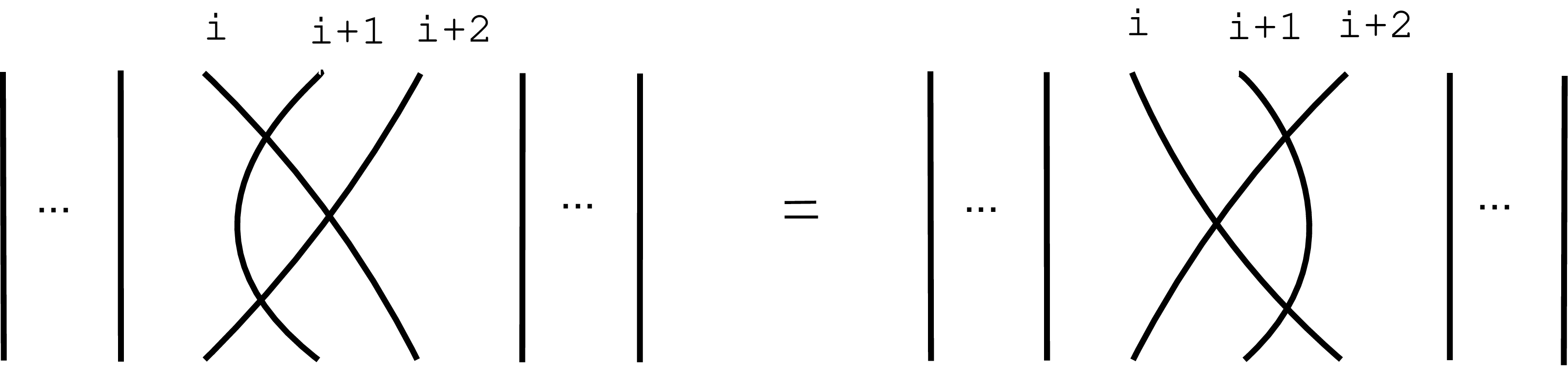}
	\caption{The relation $\ell_i \ell_{i+1} \ell_i 
= \ell_{i+1} \ell_i \ell_{i+1}$.}
\label{tripletrelation2}
\end{figure}

\begin{definition} \cite{TNaik2021}
The pure triplet group on $n$ strands, denoted by $PL_n$, is defined as the kernel of the homomorphism $L_n \longrightarrow S_n$ defined by $\ell_i \longmapsto (i \ \  i+1)$, $1\leq i \leq n-1$, where $S_n$ is the symmetric group on $n$ elements.
\end{definition}

\begin{definition}\cite{PKumar2024}
The virtual triplet group on $n$ strands, denoted by $VL_n$, is the group generated by the elements
\[
\ell_1, \ell_2, \ldots, \ell_{n-1}
\]
of the triplet group $L_n$ together with another family of elements, namely
$$\rho_1,\rho_2, \ldots, \rho_{n-1}.$$ 
In addition to the relations \eqref{eq:L1} and \eqref{eq:L2} that define $L_n$, the generators of $VL_n$ satisfy the following relations.
\begin{align}
\rho_i^2
    &= 1,
    && i=1,2,\ldots,n-1, \label{eqs1v} \\
\rho_i\rho_j
    &= \rho_j\rho_i,
    && |i-j|\ge 2, \label{eqs2v} \\
\rho_i\rho_{i+1}\rho_i
    &= \rho_{i+1}\rho_i\rho_{i+1},
    && i=1,2,\ldots,n-2, \label{eqs3v} \\
\ell_i\rho_j
    &= \rho_j\ell_i,
    && |i-j|\ge 2, \label{eqs4v} \\
\rho_i\rho_{i+1}\ell_i
    &= \ell_{i+1}\rho_i\rho_{i+1},
    && i=1,2,\ldots,n-2. \label{eqs5v}
\end{align}
\end{definition}
The generators $\rho_i$ are illustrated in Figure~\ref{vtripletgenerator}. The relations~(\ref{eqs1v}), (\ref{eqs3v}), and~(\ref{eqs5v}) are depicted in
Figures~\ref{vtripletrelation1},
\ref{vtripletrelation2}, and
\ref{vtripletrelation4}, respectively.
\begin{figure}[h]
\centering
		\includegraphics[width=5cm,height=2cm]{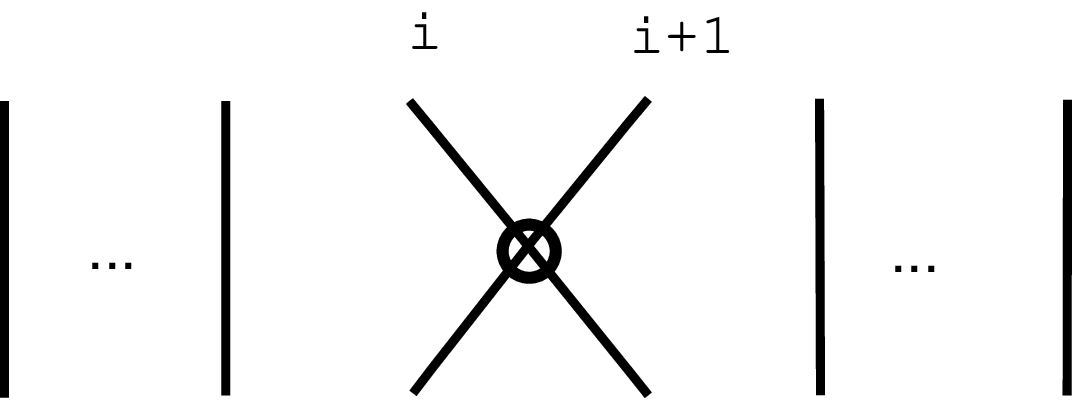}

	\caption{The generator $\rho_i.$}
\label{vtripletgenerator}
\end{figure}

\begin{figure}[h]
\centering
\includegraphics[width=10cm,height=2.2cm]{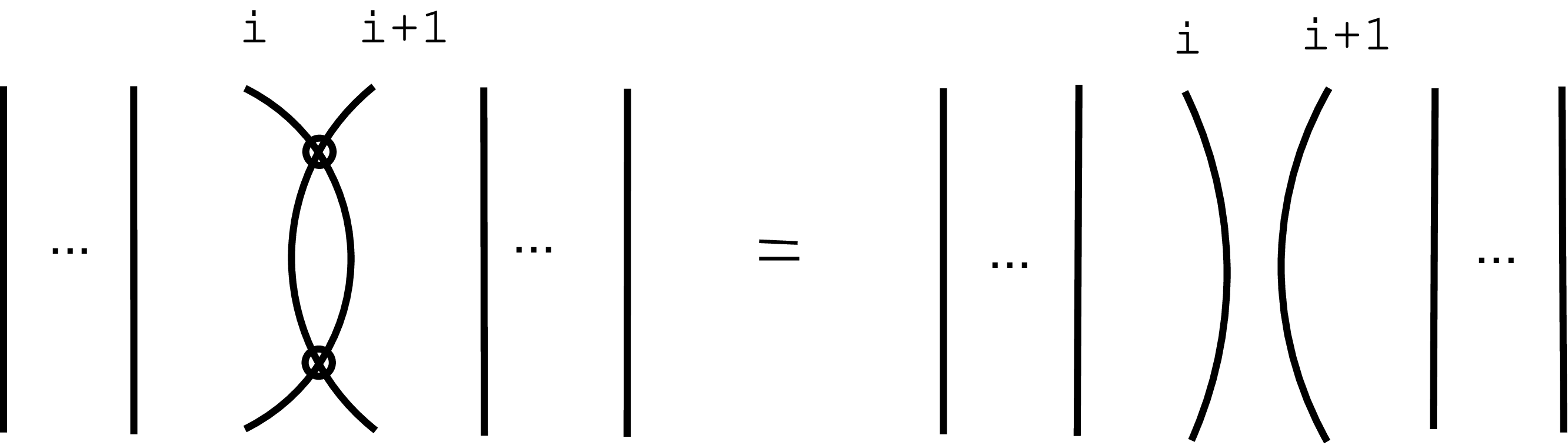}
	\caption{The relation $\rho_i^2=1$.}
\label{vtripletrelation1}
\end{figure}

\begin{figure}[h]
\centering
\includegraphics[width=10cm,height=2.2cm]{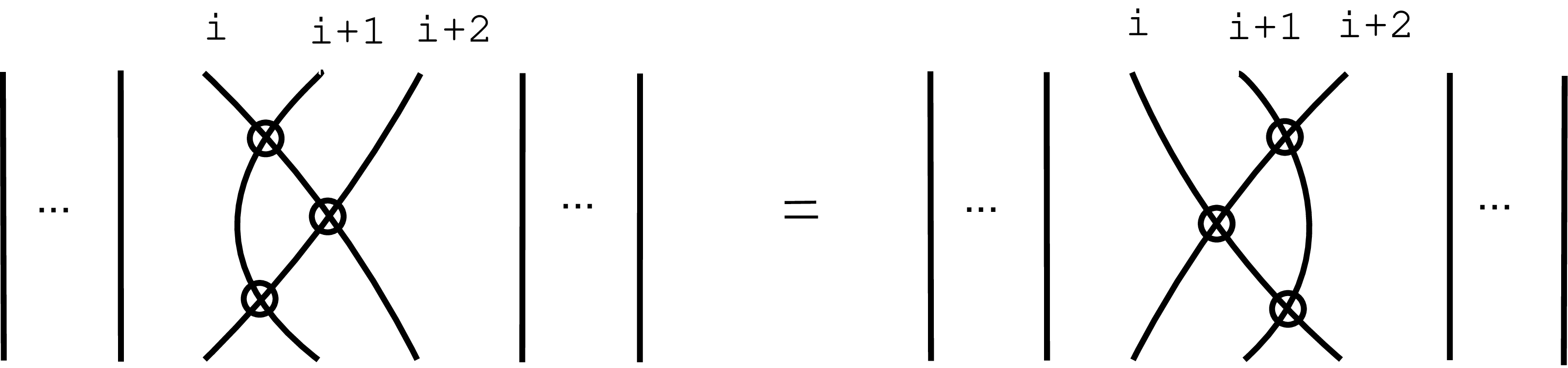}
	\caption{The relation $\rho_i \rho_{i+1} \rho_i 
= \rho_{i+1} \rho_i \rho_{i+1}$.}
\label{vtripletrelation2}
\end{figure}

\begin{figure}[h]
	\centering
\includegraphics[width=10cm,height=2.2cm]{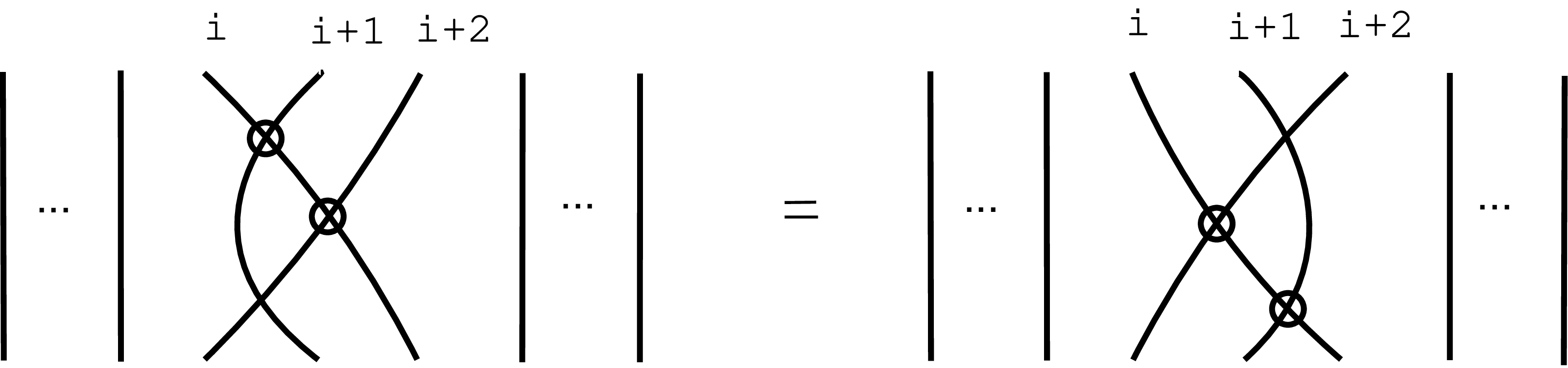}
	\caption{The relation $\rho_i \rho_{i+1} \ell_i 
= \ell_{i+1} \rho_i \rho_{i+1}$.}
\label{vtripletrelation4}
\end{figure}

We now introduce the definition of the welded triplet group, by direct analogy with the welded braid group.

\begin{definition}
The welded triplet group on $n$ strands, denoted by $WL_n$, is the group defined as the quotient of $VL_n$ by adding  the relation illustrated in
Figure~\ref{wrelation}:  
\begin{equation} \label{eqs1W}
\rho_i\ell_{i+1}\ell_i=\ell_{i+1}\ell_i\rho_{i+1}, \qquad i=1,2,\ldots,n-2.
\end{equation}
\end{definition}

\begin{figure}[h]
	\centering
\includegraphics[width=10cm,height=2.2cm]{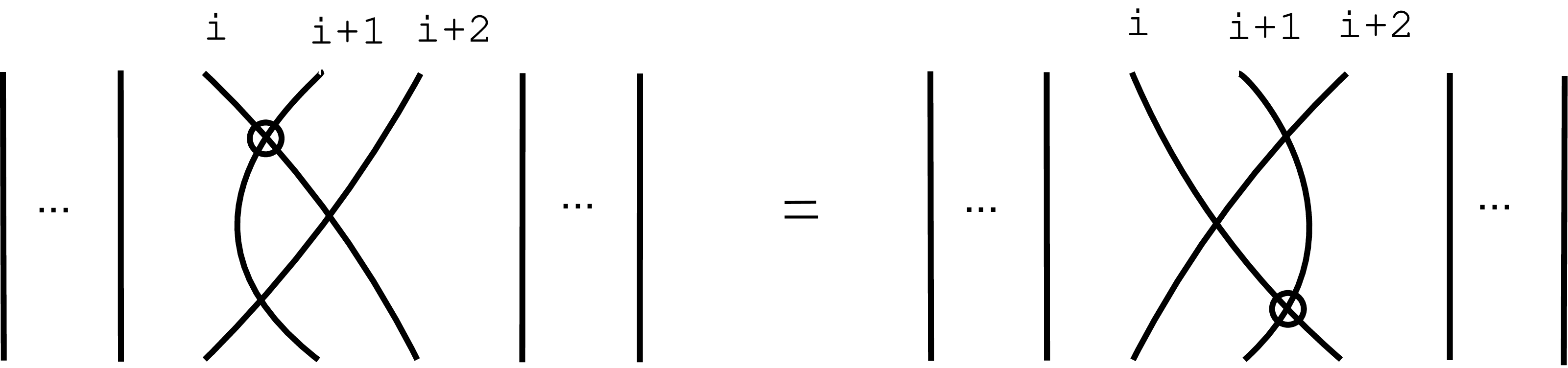}
	\caption{The welded  relation $\rho_i \ell_{i+1} \ell_i 
= \ell_{i+1} \ell_i \rho_{i+1}$.}
\label{wrelation}
\end{figure}




\section{Representations of the Triplet Group $L_n$ and Their Properties}

In this section, we study several representations of the triplet group $L_n$. We first investigate the irreducibility of the Tits representation of $L_n$ over the field $\mathbb{C}$. We then construct a new representation of $L_n$ into $\mathrm{GL}_n(\mathbb{Z}[t^{\pm 1}])$, where $t$ is an indeterminate. Additionally, we classify all complex homogeneous $2$-local representations of $L_n$ for all $n \geq 3$, as well as all complex non-homogeneous $2$-local representations of $L_3$.

\subsection{The Tits Representation of $L_n$}

\begin{definition}\cite{PKumar2024} \label{DefTits}
For $n \ge 3$, the Tits representation of the triplet group $L_n$ is the homomorphism
\[
\Theta : L_n \longrightarrow \mathrm{GL}_{n-1}(\mathbb{C}),
\qquad \Theta(\ell_i)=\Lambda_i,
\]
where the matrices $\Lambda_i$ are defined in the following way. For $l \ge 1$, let $M_l$ denote the $3\times l$ matrix whose first and third rows are zero and whose middle row consists entirely of $2$’s. For example,
\[
M_4=
\begin{pmatrix}
0 & 0 & 0 & 0\\
2 & 2 & 2 & 2\\
0 & 0 & 0 & 0
\end{pmatrix}.
\]
Let $N_l$ be the $2\times l$ matrix formed by the second and third rows of $M_l$, and let $K_l$ be the $2\times l$ matrix formed by the first and second rows of $M_l$. The matrices 
\[
\Lambda_i =
\begin{cases}
\begin{pmatrix}
\begin{pmatrix}
-1 & 1 \\
0 & 1 \\
\end{pmatrix}
& N_{n-3} \\
 0 & I_{n-3}
\end{pmatrix},
& \text{if } i=1,\\[1.2em]
\begin{pmatrix}
I_{i-2} & 0 & 0 \\
M_{i-2} &
\begin{pmatrix}
1 & 0 & 0\\
1 & -1 & 1\\
0 & 0 & 1
\end{pmatrix}
& M_{n-(i+2)} \\
0 & 0 & I_{n-(i+2)}
\end{pmatrix},
& \text{if } 2 \le i \le n-2,\\[1.2em]
\begin{pmatrix}
I_{n-3} & 0 \\
K_{n-3} &
\begin{pmatrix}
1 & 0\\
1 & -1
\end{pmatrix}
\end{pmatrix},
& \text{if } i=n-1.
\end{cases}
\]
\end{definition}

\noindent  We point out that we have defined the Tits representation of $L_n$ over $\mathbb{C}$, rather than over $\mathbb{Z}$ as in \cite{PKumar2024}, since our goal is to study its irreducibility over a field; to this end, we first introduce two lemmas.

\begin{lemma} \label{1lemma}
Let $\Theta : L_n \longrightarrow \mathrm{GL}_{n-1}(\mathbb{C})$ be the Tits representation defined in Definition~\ref{DefTits}, and let $U \subseteq \mathbb{C}^{\,n-1}$ be an invariant subspace under $\Theta$. Let $\{e_1,e_2,\ldots,e_{n-1}\}$ denote the standard basis of $\mathbb{C}^{\,n-1}$.  
If there exists an index $1 \le i \le n-1$ such that $e_i \in U$, then $U = \mathbb{C}^{\,n-1}$.
\end{lemma}

\begin{proof}
Assume first that $i=1$, that is $e_1 \in U$. Since $U$ is invariant under $\Theta$, it follows that for each $2 \leq i \leq n-1$ we have
\[
\Theta(\ell_i)(e_1) - e_1 = 2 e_i \in U.
\]
This shows that $e_i \in U$ for all $2\leq i \leq n-1$. Hence, every basis vector $e_1, e_2, \dots, e_{n-1}$ belongs to $U$, and we conclude that $ U = \mathbb{C}^{n-1},$
as required. Now, suppose that $e_i \in U$ for some $1 < i \leq n-1$. Then
\[
\Theta(\ell_1)(e_i) - e_i = 2 e_1 \in U,
\]
and so $e_1 \in U$. Returning to the first case, we again conclude that $U = \mathbb{C}^{n-1}.$
\end{proof}
\begin{lemma}\label{2lemma}
The $n \times n$ matrix
\[
A_{n} =
\begin{pmatrix}
-2 & 1 & 2 & 2 & \cdots & 2 \\
1 & -2 & 1 & 2 & \cdots & 2 \\
2 & 1 & -2 & 1 & \cdots & 2 \\
\vdots & \ddots & \ddots & \ddots & \ddots & \vdots \\
2 & \cdots & 2 & 1 & -2 & 1 \\
2 & \cdots & 2 & 2 & 1 & -2
\end{pmatrix}
\]
is an invertible matrix for every positive integer $n$.
\end{lemma}

\begin{proof}
The result follows from standard linear algebra arguments; we omit the routine details.
\end{proof}

\begin{theorem} \label{Tits irr}
The Tits representation $\Theta : L_n \longrightarrow \mathrm{GL}_{n-1}(\mathbb{C})$ given in Definition~\ref{DefTits} is irreducible.
\end{theorem}

\begin{proof}
Suppose, toward a contradiction, that the Tits representation is reducible. 
Then there exists a nontrivial proper subspace 
\( U \subset \mathbb{C}^{n-1} \) that is invariant under \( \Theta \).
Choose a nonzero vector 
\( u = (u_1, u_2, \ldots, u_{n-1})^{T} \in U \).
Since \( U \) is \( \Theta \)-invariant, we have \( \Theta(\ell_i)u-u \in U \)
for every $1\leq i \leq n-1$. Hence, the following hold.\\
$\Theta(\ell_1)(u)-u=(-2u_1+u_2+2u_3+2u_4+\ldots+2u_{n-1})e_1\in U$,\\
$\Theta(\ell_2)(u)-u=(u_1-2u_2+u_3+2u_4+2u_5+\ldots+2u_{n-1})e_2\in U$,\\
$\Theta(\ell_3)(u)-u=(2u_1+u_2-2u_3+u_4+2u_5+2u_6+\ldots+2u_{n-1})e_3\in U$,\\
\hspace*{0.4cm} \vdots\\
$\Theta(\ell_{n-2})(u)-u=(2u_1+2u_2+\ldots+2u_{n-4}+u_{n-3}-2u_{n-2}+u_{n-1})e_{n-2}\in U$,\\
$\Theta(\ell_{n-1})(u)-u=(2u_1+2u_2+\ldots+2u_{n-3}+u_{n-2}-2u_{n-1})e_{n-1}\in U$.\\
By Lemma~\ref{1lemma}, none of the standard basis vectors 
\( e_i \), \( 1 \le i \le n-1 \), belongs to \( U \) as \( U \) is nontrivial. This leads to the following system of equations.

\[
\begin{cases}
-2u_1 + u_2 + 2u_3 + 2u_4 + \cdots + 2u_{n-1} = 0, \\[4pt]
u_1 - 2u_2 + u_3 + 2u_4 + 2u_5 + \cdots + 2u_{n-1} = 0, \\[4pt]
2u_1 + u_2 - 2u_3 + u_4 + 2u_5 + 2u_6 + \cdots + 2u_{n-1} = 0, \\[4pt]
\hspace*{0.5cm}\vdots \\[4pt]
2u_1 + 2u_2 + \cdots + 2u_{n-4} + u_{n-3} - 2u_{n-2} + u_{n-1} = 0, \\[4pt]
2u_1 + 2u_2 + \cdots + 2u_{n-3} + u_{n-2} - 2u_{n-1} = 0 .
\end{cases}
\]
Transforming this system to a matrix equation, we get that 
$$A_{n-1}u=0,$$
where $A_{n-1}$ is the $(n-1)\times(n-1)$ matrix given by
\[
A_{n-1} =
\begin{pmatrix}
-2 & 1 & 2 & 2 & \cdots & 2 \\
1 & -2 & 1 & 2 & \cdots & 2 \\
2 & 1 & -2 & 1 & \cdots & 2 \\
\vdots & \ddots & \ddots & \ddots & \ddots & \vdots \\
2 & \cdots & 2 & 1 & -2 & 1 \\
2 & \cdots & 2 & 2 & 1 & -2
\end{pmatrix}.
\]
By Lemma \ref{2lemma}, we know that \( A_{n-1} \) is invertible.
Consequently, the matrix equation \( A_{n-1}u = 0 \) has only the zero solution,
contradicting the assumption that \( u \) is a nonzero vector.
Therefore, the Tits representation \( \Theta \) is irreducible, as desired.
\end{proof}

\subsection{Constructing a New Representation of $L_n$}

Let $t$ be an indeterminate and $k \in \mathbb{Z}$. Let $x_1, x_2, \ldots, x_n$ be the standard basis of $\mathbb{Z}[t^{\pm 1}]^n$. Define a mapping
\[
\mu : L_n \longrightarrow \mathrm{GL}_n(\mathbb{Z}[t^{\pm 1}])
\]
by
\[
\mu(\ell_i):
\begin{cases}
x_i \longmapsto t^kx_{i+1},\\
x_{i+1} \longmapsto t^{-k}x_{i},\\
x_j \longmapsto x_j \quad \text{for } j \notin \{ i,i+1 \}.
\end{cases}
\]

\begin{theorem} \label{defmuu}
The mapping $\mu$ defines a representation of $L_n$.
\end{theorem}

\begin{proof}
To prove that $\mu$ defines a representation of $L_n$, it suffices to show that $\mu$ preserves the defining relations of $L_n$. For $1 \leq i \leq n-1$, we examine the two relations of $L_n$ in the following cases.
\begin{itemize}
    \item [(1)] \underline{$\mu(\ell_i)^2 = 1$}.\vspace{0.1cm} \\ To prove this, we should verify that $
\mu(\ell_i)^2(x_j) = x_j$ for all $1 \leq j \leq n.$ 
For this purpose, we consider, in the following, the three cases in order: $j = i$, $j = i+1$, and $j \notin \{ i,i+1 \}$.\vspace{0.1cm}
    \begin{itemize}
        \item[$\bullet$] $\mu(\ell_i)^2(x_i)= \mu(\ell_i)(t^kx_{i+1})= t^k\mu(\ell_i)(x_{i+1})=
        t^k(t^{-k}x_i)=x_i.$\vspace{0.1cm}
        \item[$\bullet$] $\mu(\ell_i)^2(x_{i+1})= \mu(\ell_i)(t^{-k}x_{i})=t^{-k}\mu(\ell_i)(x_{i})=
        t^{-k}(t^{k}x_{i+1})=x_{i+1}.$\vspace{0.1cm}
         \item[$\bullet$] $\mu(\ell_i)^2(x_{j})=\mu(\ell_i)(x_{j})=x_j$ for all $j \notin \{ i,i+1 \}.$
    \end{itemize}
    Thus, the mapping $\mu$ preserves the relation $\ell_i^2=1$ of $L_n$, as required.\vspace{0.1cm}
    \item[(2)] 
\underline{$\mu(\ell_i)\,\mu(\ell_{i+1})\,\mu(\ell_i) = \mu(\ell_{i+1})\,\mu(\ell_i)\,\mu(\ell_{i+1})$}.\vspace{0.1cm}
\\ 
Again, it is enough to prove that 
$
\mu(\ell_i)\,\mu(\ell_{i+1})\,\mu(\ell_i)(x_j) = \mu(\ell_{i+1})\,\mu(\ell_i)\,\mu(\ell_{i+1})(x_j)$ 
for all $1 \leq j \leq n$. To this end, we consider the following four cases: $j=i$, $j=i+1$, $j=i+2$, and $j \notin \{ i,i+1,i+2 \}$. These cases are treated successively below.
    \begin{itemize}
        \item[$\bullet$] First, we consider the left hand side of the equation. We have:\\  $\mu(\ell_i)\mu(\ell_{i+1})\mu(\ell_i)(x_i)= \mu(\ell_i)\mu(\ell_{i+1})(t^kx_{i+1})\\ \hspace*{3.19cm}=
        t^k\mu(\ell_i)\mu(\ell_{i+1})(x_{i+1})\\ \hspace*{3.19cm}=t^k\mu(\ell_i)(t^kx_{i+2})\\ \hspace*{3.19cm}=t^{2k}\mu(\ell_i)(x_{i+2})\\ \hspace*{3.19cm} =t^{2k}x_{i+2}.$\\
        Now, for the right hand side we have:\\
        $\mu(\ell_{i+1})\mu(\ell_{i})\mu(\ell_{i+1})(x_i)= \mu(\ell_{i+1})\mu(\ell_{i})(x_{i})\\ \hspace*{3.55cm}=  \mu(\ell_{i+1})(t^kx_{i+1})\\ \hspace*{3.55cm}= t^k\mu(\ell_{i+1})(x_{i+1})\\ \hspace*{3.55cm}=t^{k}(t^kx_{i+2})\\ \hspace*{3.55cm}=t^{2k}x_{i+2}$.\\
A comparison of the left hand side and the right hand side shows that they are equal.
        \item[$\bullet$] Also, we consider first the left hand side of the equation. We have:\\  $\mu(\ell_i)\mu(\ell_{i+1})\mu(\ell_i)(x_{i+1})= \mu(\ell_i)\mu(\ell_{i+1})(t^{-k}x_{i})\\ \hspace*{3.54cm}=
        t^{-k}\mu(\ell_i)\mu(\ell_{i+1})(x_{i})\\ \hspace*{3.54cm}=t^{-k}\mu(\ell_i)(x_{i})\\ \hspace*{3.54cm}=t^{-k}(t^kx_{i+1})\\ \hspace*{3.54cm} =x_{i+1}.$\\
        Now, for the right hand side we have:\\
        $\mu(\ell_{i+1})\mu(\ell_{i})\mu(\ell_{i+1})(x_{i+1})= \mu(\ell_{i+1})\mu(\ell_{i})(t^kx_{i+2})\\ \hspace*{3.9cm}=  t^k\mu(\ell_{i+1})\mu(\ell_{i})(x_{i+2})\\ \hspace*{3.9cm}= t^k\mu(\ell_{i+1})(x_{i+2})\\ \hspace*{3.9cm}=t^{k}(t^{-k}x_{i+1})\\ \hspace*{3.9cm}=x_{i+1}$.\\
Again, the left hand side and the right hand side are equal.
        \item[$\bullet$] Similarly, we consider first the left hand side of the equation. We have:\\  $\mu(\ell_i)\mu(\ell_{i+1})\mu(\ell_i)(x_{i+2})= \mu(\ell_i)\mu(\ell_{i+1})(x_{i+2})\\ \hspace*{3.54cm}=\mu(\ell_i)(t^{-k}x_{i+1})\\ \hspace*{3.54cm}=
        t^{-k}\mu(\ell_i)(x_{i+1})\\ \hspace*{3.54cm}=t^{-k}(t^{-k}x_{i})\\ \hspace*{3.54cm} =t^{-2k}x_{i}.$\\
        For the right hand side we have:\\
        $\mu(\ell_{i+1})\mu(\ell_{i})\mu(\ell_{i+1})(x_{i+2})= \mu(\ell_{i+1})\mu(\ell_{i})(t^{-k}x_{i+1})\\ \hspace*{3.9cm}=  t^{-k}\mu(\ell_{i+1})\mu(\ell_{i})(x_{i+1})\\ \hspace*{3.9cm}= t^{-k}\mu(\ell_{i+1})(t^{-k}x_{i})\\ \hspace*{3.9cm}=t^{-2k}\mu(\ell_{i+1})(x_{i})\\ \hspace*{3.9cm}=t^{-2k}x_{i}$.\\
We get also the left hand side and the right hand side are equal.
\item[$\bullet$] When $j \notin \{ i,i+1,i+2 \}$, we have $\mu(\ell_i)(x_j) = \mu(\ell_{i+1})(x_j) = x_j$. Hence, the relation is trivially satisfied in this case.
    \end{itemize}
    Thus, the mapping $\mu$ preserves the relation $\ell_i \ell_{i+1} \ell_i 
= \ell_{i+1} \ell_i \ell_{i+1}$ of $L_n$, as desired.
\end{itemize}
Therefore, the mapping $\mu$ defines a representation of $L_n$, as required.
\end{proof}

In what follows, we consider the matrix representation of $L_n$ associated
with the representation $\mu$. We denote it by $\mu'$ to distinguish it
from the module representation $\mu$.

\begin{definition}\label{propmu}
The matrix representation of $L_n$ induced by $\mu$ is the representation 
$$\mu':L_n\longrightarrow \mathrm{GL}_n(\mathbb{Z}[t^{\pm 1}])$$
given by sending each generator $\ell_i$, for $1\le i\le n-1$, to the block matrix
\[
\begin{pmatrix}
I_{i-1} & 0 & 0 \\
0 &
\begin{pmatrix}
0 & t^k \\
t^{-k} & 0
\end{pmatrix}
& 0 \\
0 & 0 & I_{n-i-1}
\end{pmatrix}.
\]
\end{definition}


The following theorem provides a complete study of the faithfulness of the representation $\mu'$.

\begin{theorem}\label{Thmfa}
The representation $\mu'$ is faithful for $n=2$ and $n=3$, and unfaithful for all $n\geq 4$.
\end{theorem}

\begin{proof}
We divide the proof into three cases as follows.\vspace{0.05cm}
\begin{itemize}
    \item[(1)] \underline{The case $n=2$}: In this case, we have $L_2=\{e,\ell_1 \}$ and 
    $$\mu'(\ell_1)=\begin{pmatrix}
0 & t^k \\
t^{-k} & 0
\end{pmatrix}\neq I_2.$$
Hence, clearly $\mu'$ is faithful.\vspace{0.1cm}
\item[(2)] \underline{The case $n=3$}: In this case, we have $L_3=\{e,\ell_1, \ell_2, \ell_1\ell_2, \ell_2\ell_1, \ell_1\ell_2\ell_1\},$
where their images under $\mu'$ are given as follows.
 $$\mu'(\ell_1)=\begin{pmatrix}
0 & t^k & 0\\
t^{-k} & 0 & 0\\
0 & 0 & 1
\end{pmatrix}\neq I_3,$$
$$
\mu'(\ell_2)=\left(
\begin{array}{ccc}
 1 & 0 & 0 \\
 0 & 0 & t^k \\
 0 & t^{-k} & 0 \\
\end{array}
\right)\neq I_3,
$$
$$
\mu'(\ell_1\ell_2)=\left(
\begin{array}{ccc}
 0 & 0 & t^{2 k} \\
 t^{-k} & 0 & 0 \\
 0 & t^{-k} & 0 \\
\end{array}
\right)\neq I_3,
$$
$$
\mu'(\ell_2\ell_1)=\left(
\begin{array}{ccc}
 0 & t^k & 0 \\
 0 & 0 & t^k \\
 t^{-2 k} & 0 & 0 \\
\end{array}
\right)\neq I_3,
$$
and
$$
\mu'(\ell_1\ell_2\ell_1)=\left(
\begin{array}{ccc}
 0 & 0 & t^{2 k} \\
 0 & 1 & 0 \\
 t^{-2 k} & 0 & 0 \\
\end{array}
\right)\neq I_3.
$$
Thus, we get that $\mu'$ is faithful in this case also.\vspace{0.1cm}
\item[(3)] \underline{The case $n \geq 4$}: For all $1 \leq i,j \leq n-1$ with $|i-j| \geq 2$, a direct computation shows that $\mu'(\ell_i)\mu'(\ell_j)=\mu'(\ell_j)\mu'(\ell_i).$
However, since the commutative relation $\ell_i\ell_j=\ell_j\ell_i$ does not hold in $L_n$, it follows that the representation $\mu'$ is unfaithful in this case.
\end{itemize}
\end{proof}

\begin{proposition}
For $n \geq 4$, the kernel of the representation
\[
\mu' : L_n \longrightarrow \mathrm{GL}_n(\mathbb{Z}[t^{\pm 1}])
\]
constructed in Definition \ref{propmu} coincides with the pure triplet group $PL_n$.
\end{proposition}

\begin{proof}
Recall that the pure triplet group $PL_n$ is defined as the kernel of the epimorphism 
\[
\pi : L_n \longrightarrow S_n, \quad \pi(\ell_i) = (i \ \ i+1), \quad 1\leq i \leq n-1.
\] 
We now define a representation, namely $f$, of $S_n$ into $\mathrm{GL}_n(\mathbb{Z}[t^{\pm 1}])$ by setting 
\[
f(i \ \ i+1) = \mu'(\ell_i), \quad 1\leq i \leq n-1.
\] 
It is evident that the representation $\mu'$ of $L_n$ factors through this representation of $S_n$, that is, 
\[
\mu'(x) = (f \circ \pi)(x) \quad \text{for all } x \in L_n.
\] 
Moreover, one can verify that the images of the generators of $S_n$ under $f$ generate a subgroup of order $n!$ in $\mathrm{GL}_n(\mathbb{Z}[t^{\pm 1}])$. Hence, the representation $f$ is faithful, which establishes the claim.
\end{proof}

We now investigate the irreducibility of the representation $\mu'$.

\begin{theorem} \label{Thmir}
The representation $\mu'$ is reducible for all $n\geq 2$.
\end{theorem}

\begin{proof}
We consider two cases for the proof.
\begin{itemize}
    \item [(1)] \underline{The case $k=0$}: In this case, we can directly see that the column vector $(1,1,\ldots, 1)^T$ is invariant under $\mu'(\ell_i)$ for all $1\leq i \leq n-1$, and so $\mu'$ is reducible. \vspace{0.1cm}
 \item [(2)] \underline{The case $k\neq 0$}: We introduce in this case an equivalent representation of $\mu'$, namely $\mu''$, in the following manner. Let 
$$P = \mathrm{diag}\left(t^{k(n-1)}, t^{k(n-2)}, \ldots, t^k, 1\right)$$ be the $n \times n$ diagonal matrix with diagonal entries $t^{k(n-1)}, t^{k(n-2)}, \ldots, t^k,1$. Define a new representation $\mu''$ of $L_n$ by 
$$\mu'' = P^{-1}\mu'P.$$ 
A direct matrix computation shows that the action of $\mu''$ on the generators $\ell_i$, $1 \leq i \leq n-1$, of $L_n$ is given by the following.
$$\mu''(\ell_i)=
\left( \begin{array}{c|@{}c|c@{}}
   \begin{matrix} I_{i-1} \end{matrix} & 0 & 0 \\
   \hline
   0 & \hspace{0.2cm} \begin{matrix} 0 & 1\\ 1 & 0 \end{matrix} & 0 \\
   \hline
   0 & 0 & I_{n-i-1} 
\end{array} \right).$$ Similarly, in this case, the column vector $(1,1,\ldots,1)^T$ is invariant under $\mu''(\ell_i)$ for all $1 \leq i \leq n-1$. Therefore, $\mu''$ is reducible, and consequently, $\mu'$ is also reducible.
\end{itemize}
\end{proof}

\subsection{The Complex Homogeneous $2$-Local Representations of $L_n$}

\begin{definition}\cite{MNasser20241}
Let $G$ be a group generated by $g_1, g_2, \ldots, g_{n-1}$. A representation $\lambda: G \longrightarrow \mathrm{GL}_m(\mathbb{C})$ is said to be \emph{$k$-local} if, for each $1 \leq i \leq n-1$, the image of $g_i$ has the block diagonal form
\[
\lambda(g_i) =
\begin{pmatrix}
I_{i-1} & 0 & 0 \\
0 & M_i & 0 \\
0 & 0 & I_{n-i-1}
\end{pmatrix},
\]
where $M_i \in \mathrm{GL}_k(\mathbb{C})$, $k = m - n + 2$, and $I_r$ denotes the $r \times r$ identity matrix. The representation is called homogeneous if all the matrices $M_i$ are identical.
\end{definition}

The classification of complex homogeneous $2$-local representations of $L_n$ is of particular interest, since the far commutativity relations of the braid group $B_n$ do not hold in $L_n$. In \cite{yMikhalchishina2013}, Y.~Mikhalchishina classified all homogeneous $2$-local representations of $B_n$ for $n \ge 3$, as well as all $2$-local representations of $B_3$. More recently, in \cite{TMayassi2025}, T.~Mayassi and M.~Nasser classified all homogeneous $3$-local representations of $B_n$ for $n \ge 4$. These results motivate an analogous classification for $L_n$, where the absence of far commutativity leads to fundamentally different representation behavior.

\begin{theorem} \label{TH13}
Fix $n \geq 3$, and let $\lambda \colon L_n \longrightarrow \mathrm{GL}_n(\mathbb{C})$ be a non-trivial complex homogeneous $2$-local representation of $L_n$. Then $\lambda$ is given by acting on the generators $\ell_i$, $1 \leq i \leq n-1$, of $L_n$, as described below.
\[
\lambda(\ell_i)=
\begin{pmatrix}
I_{i-1} & 0 & 0 \\
0 &
\begin{pmatrix}
0 & b \\
\dfrac{1}{b} & 0
\end{pmatrix}
& 0 \\
0 & 0 & I_{n-i-1}
\end{pmatrix},
\]
where $b \in \mathbb{C}^*$.
\end{theorem}

\begin{proof}
Since $\lambda$ is a non-trivial complex homogeneous $2$-local representation of $L_n$, we may write
\[
\lambda(\ell_i)=
\begin{pmatrix}
I_{i-1} & 0 & 0 \\
0 &
\begin{pmatrix}
a & b \\
c & d
\end{pmatrix}
& 0 \\
0 & 0 & I_{n-i-1}
\end{pmatrix},
\]
where $a,b,c,d \in \mathbb{C}$ satisfy $ad-bc \neq 0$, and the inner $2\times 2$ block is not the identity matrix $I_2$. Now, from the structure of homogeneous $2$-local representations of $L_n$, it suffices to consider the following two defining relations of $L_n$, since all other relations yield similar constraints:
\[
\ell_1^2 = 1 \quad \text{and} \quad \ell_1\ell_2\ell_1 = \ell_2\ell_1\ell_2.
\]
Applying these relations to the image under $\lambda$, we obtain the following system of nine equations and four unknowns.
\begin{equation}\label{eqq10}
a^2 + bc - 1 = 0,
\end{equation}
\begin{equation}\label{eqq11}
b(a+d) = 0,
\end{equation}
\begin{equation}\label{eqq12}
c(a+d) = 0,
\end{equation}
\begin{equation}\label{eqq13}
bc + d^2 - 1 = 0,
\end{equation}
\begin{equation}\label{eqq14}
a(-1 + a + bc) = 0,
\end{equation}
\begin{equation}\label{eqq15}
abd = 0,
\end{equation}
\begin{equation}\label{eqq16}
acd = 0,
\end{equation}
\begin{equation}\label{eqq17}
ad(-a + d) = 0,
\end{equation}
\begin{equation}\label{eqq18}
d(1 - bc - d^2) = 0.
\end{equation}

The system naturally divides into three cases:
\begin{itemize}
   \item[$\bullet$] \underline{Case $a \neq 0$ and $d \neq 0$}:  
    Equations \eqref{eqq15} and \eqref{eqq16} force $b = c = 0$. Substituting this into Equation \eqref{eqq10} yields $a^2 = 1$. On the other hand, Equation \eqref{eqq17} implies $a = d$.  Hence $a = d = \pm 1$. The choice $a = d = 1$ corresponds to the trivial representation and is therefore excluded, leaving $a = d = -1$. This value of such $a$ does not represent a solution for Equation \eqref{eqq14}. So, this case does not occur.
    \item[$\bullet$] \underline{Case $a = 0$}:  
    Equation \eqref{eqq10} gives $bc = 1$, and then Equation \eqref{eqq13} yields $d = 0$. This gives our desired result. \vspace{0.1cm}
    \item[$\bullet$] \underline{Case $d = 0$}:  
    Equation \eqref{eqq13} gives $bc = 1$, and then Equation \eqref{eqq10} yields $a = 0$. This again gives our desired result.
\end{itemize}
\end{proof}

The previous theorem implies the following result.

\begin{corollary}
Any non-trivial complex homogeneous $2$-local representation $$\lambda \colon L_n \longrightarrow \mathrm{GL}_n(\mathbb{C})$$ is equivalent to the complex specialization of the representation $$\mu':L_n\longrightarrow \mathrm{GL}_n(\mathbb{Z}[t^{\pm 1}]).$$
\end{corollary}

\begin{remark}
In general, the complex specialization of an arbitrary representation
$\phi : G \longrightarrow \mathrm{GL}_n(\mathbb{Z}[t^{\pm1}])$ does not necessarily preserve
properties such as faithfulness or irreducibility.  
However, by adapting the arguments used in the proofs of
Theorems~\ref{Thmfa} and~\ref{Thmir}, one can obtain analogous conclusions for
every non-trivial complex homogeneous $2$-local representation of $L_n$.
In particular, the representation
\[
\lambda \colon L_n \longrightarrow \mathrm{GL}_n(\mathbb{C})
\]
determined in Theorem \ref{TH13} is faithful for $n=2,3$ and unfaithful for all $n \ge 4$. Moreover, it is
reducible for all $n \ge 3$.
\end{remark}

\subsection{The Complex Non-Homogeneous $2$-Local Representations of $L_3$}

In this subsection, we determine all non-trivial complex non-homogeneous $2$-local representations of the triplet group $L_3=\langle\ell_1,\ell_2\rangle$.

\begin{theorem}\label{Thhh222}
Let $\zeta:L_3\longrightarrow\mathrm{GL}_3(\mathbb{C})$ be a non-trivial $2$-local representation of $L_3$ (not necessarily homogeneous).
Then $\zeta$ is equivalent to exactly one of the following four representations, given by pairs
\[
\zeta(\ell_1)=L_1\ \text{and }\ \zeta(\ell_2)=L_2,
\]
where the matrices $L_1$ and $L_2$ are of the form
$$
L_1=
\begin{pmatrix}
a & b & 0\\
c & d & 0\\
0 & 0 & 1
\end{pmatrix},$$
$$
L_2=
\begin{pmatrix}
1 & 0 & 0\\
0 & e & f\\
0 & g & h
\end{pmatrix},
$$
and belong to one of the following families.

\begin{itemize}
\item[(1)]
\[
L_1^{(1)}=
\begin{pmatrix}
\dfrac{e}{1-e} & b & 0\\[2pt]
\dfrac{1-2 e}{b (e-1)^2} & \dfrac{e}{e-1} & 0\\
0 & 0 & 1
\end{pmatrix} \ \text{ and} \ \
L_2^{(1)}=
\begin{pmatrix}
1 & 0 & 0\\[2pt]
0 & e & f\\[6pt]
0 & \dfrac{1-e^2}{f} & -e
\end{pmatrix},
\]
where $b,e,f\in \mathbb{C}, b\neq 0, e\neq 1, f\neq 0$.
\item[(2)]
\[
L_1^{(2)}=
\begin{pmatrix}
-\tfrac{1}{2} & b & 0\\
\dfrac{3}{4 b} & \dfrac{1}{2} & 0\\
0 & 0 & 1
\end{pmatrix} \ \text{ and} \ \
L_2^{(2)}=
\begin{pmatrix}
1 & 0 & 0\\
0 & -1 & 0\\
0 & g & 1
\end{pmatrix},
\]
where $b,g\in \mathbb{C}$, $b\neq 0$.
\item[(3)]
\[
L_1^{(3)}=
\begin{pmatrix}
1 & 0 & 0\\[2pt]
c & -1 & 0\\
0 & 0 & 1
\end{pmatrix} \ \text{ and} \ \
L_2^{(3)}=
\begin{pmatrix}
1 & 0 & 0\\
0 & \tfrac{1}{2} & f\\
0 & \dfrac{3}{4 f} & -\dfrac{1}{2}
\end{pmatrix},
\]
where $c,f \in \mathbb{C}, f\neq 0.$
\item[(4)]
\[
L_1^{(4)}=
\begin{pmatrix}
1 & 0 & 0\\
0 & -1 & 0\\
0 & 0 & 1
\end{pmatrix} \ \text{ and} \ \
L_2^{(4)}=
\begin{pmatrix}
1 & 0 & 0\\
0 & -1 & 0\\
0 & 0 & 1
\end{pmatrix}.
\]
\end{itemize}
\end{theorem}

\begin{proof}
Recall that $$L_3=\langle\ell_1,\ell_2\ | \ \ell_1^2=\ell^2_2=1,\ell_1\ell_2\ell_1=\ell_2\ell_1\ell_2\rangle.$$ 
Applying the representation $\zeta$ to the defining relations of $L_3$, we obtain
$$\zeta(\ell_1)^2=\zeta(\ell_2)^2=I_3 \ \text{ and }\ \zeta(\ell_1)\zeta(\ell_2)\zeta(\ell_1)=\zeta(\ell_2)\zeta(\ell_1)\zeta(\ell_2).$$
These identities impose algebraic constraints on the matrix entries of
$\zeta(\ell_1)$ and $\zeta(\ell_2)$, leading to the following system of seventeen equations and eight unknowns.
\begin{equation}
-1+a^2+bc=0,
\end{equation}
\begin{equation}
b(a+d))=0,
\end{equation}
\begin{equation}
c(a+d)=0,
\end{equation}
\begin{equation}
-1 + b c + d^2=0,
\end{equation}
\begin{equation}
-1 + e^2 + f g=0,
\end{equation}
\begin{equation}
f(e+h)=0,
\end{equation}
\begin{equation}
g(e+h)=0,
\end{equation}
\begin{equation}
-1 + f g + h^2=0,
\end{equation}
\begin{equation}
-a + a^2 + b c e=0,
\end{equation} 
\begin{equation}
b(a-e+de)=0,
\end{equation}
\begin{equation}
c(a-e+de)=0,
\end{equation}
\begin{equation}
b c + d^2 e - d e^2 - f g=0,
\end{equation}
\begin{equation}
f(d-de-h)=0,
\end{equation}
\begin{equation}
g(d-de-h)=0,
\end{equation}
\begin{equation}
-d f g + h - h^2=0,
\end{equation}
\begin{equation}
ad-bc\neq 0,
\end{equation}
\begin{equation}
eh-fg\neq 0.
\end{equation}
Solving this system of equations, following the approach of
Theorem~\ref{TH13}, yields the desired result.
\end{proof}

\section{Extensions of Representations of $L_n$ to $VL_n$ and $WL_n$}

This section studies how representations of the triplet group $L_n$ can be extended to $VL_n$ and $WL_n$, focusing on complex $2$-local representations. We classify these extensions and analyze their faithfulness and irreducibility.

\begin{proposition} \label{prop11}
Let $n \geq 3$ and let $\lambda \colon L_n \longrightarrow \mathrm{GL}_n(\mathbb{C})$ be a complex $2$-local representation of $L_n$. Define two maps as follows:
\begin{itemize}
    \item[(1)] $\hat{\lambda}_1 \colon VL_n \longrightarrow \mathrm{GL}_n(\mathbb{C})$ given by
    \[
    \hat{\lambda}_1(\ell_i)=\lambda(\ell_i)
    \quad \text{and} \quad
    \hat{\lambda}_1(\rho_i)=\lambda(\ell_i),
    \qquad 1 \leq i \leq n-1.
    \]
    \item[(2)] $\hat{\lambda}_2 \colon VL_n \longrightarrow \mathrm{GL}_n(\mathbb{C})$ given by
    \[
    \hat{\lambda}_2(\ell_i)=\lambda(\ell_i)
    \quad \text{and} \quad
    \hat{\lambda}_2(\rho_i)=-\lambda(\ell_i),
    \qquad 1 \leq i \leq n-1.
    \]
\end{itemize}
Then both $\hat{\lambda}_1$ and $\hat{\lambda}_2$ are well-defined representations of $VL_n$. Moreover, each of them induces a representation of $WL_n$.
\end{proposition}

\begin{proof}
We verify that the two maps $\hat{\lambda}_1$ and $\hat{\lambda}_2$ satisfy the defining relations \eqref{eqs1v}, \eqref{eqs2v}, \ldots, \eqref{eqs5v} of $VL_n$, as well as the defining relation \eqref{eqs1W} of $WL_n$. We have the following facts.
\begin{itemize}
\item[$\bullet$] First, the relations \eqref{eqs1v}, \eqref{eqs3v}, and \eqref{eqs5v} in the presentation of $VL_n$, together with the relation \eqref{eqs1W} defining $WL_n$, are satisfied automatically. Indeed, this follows directly from the fact that
\[
\hat{\lambda}_1(\rho_i)=\lambda(\ell_i)
\quad \text{and} \quad
\hat{\lambda}_2(\rho_i)=-\lambda(\ell_i),
\]
combined with the fact that $\lambda$ is a representation of $L_n$.
\vspace{0.1cm}
\item[$\bullet$] Second, relations \eqref{eqs2v} and \eqref{eqs4v} follow directly from the fact that $\lambda$ is a $2$-local representation of $L_n$.
\end{itemize}
\end{proof}

The previous proposition shows that every $2$-local representation $\lambda$ of $L_n$ can be extended to $VL_n$, and consequently to $WL_n$. Therefore, the existence of such extensions is established. We refer to this type of extension as the \emph{standard type}. Note that the condition that $\lambda$ is a $2$-local representation is essential for this type of extension; otherwise, the relations \eqref{eqs2v} and \eqref{eqs4v} of $VL_n$ would not be satisfied.\vspace{0.1cm}

We now consider the non-trivial complex homogeneous $2$-local representations of $L_n$ described in Theorem \ref{TH13}. Our aim is to extend these representations to $VL_n$ and $WL_n$ and study their properties. In fact, we classify all non-trivial complex homogeneous $2$-local representations of $VL_n$ and $WL_n$ for all $n \geq 3$.

\begin{theorem} \label{Them111}
Fix $n \geq 3$, and let $\omega \colon VL_n \longrightarrow \mathrm{GL}_n(\mathbb{C})$ be a non-trivial complex homogeneous $2$-local representation of $VL_n$. Then $\omega$ is equivalent to one of the following two representations $\omega_j, 1\leq j\leq 2,$ acting on the generators $\ell_i$ and $\rho_i$, $1 \leq i \leq n-1$, of $VL_n$, as described below.
\begin{itemize}
    \item[(1)] $\omega_1 \colon VL_n \longrightarrow \mathrm{GL}_n(\mathbb{C})$ 
\[
\omega_1(\ell_i)=
\begin{pmatrix}
I_{i-1} & 0 & 0 \\
0 &
\begin{pmatrix}
0 & b \\
\dfrac{1}{b} & 0
\end{pmatrix}
& 0 \\
0 & 0 & I_{n-i-1}
\end{pmatrix} \text{ and } \omega_1(\rho_i)=
\begin{pmatrix}
I_{i-1} & 0 & 0 \\
0 &
\begin{pmatrix}
0 & x \\
\dfrac{1}{x} & 0
\end{pmatrix}
& 0 \\
0 & 0 & I_{n-i-1}
\end{pmatrix} ,
\]
where $b,x \in \mathbb{C}^*$.
    \item[(2)] $\omega_2 \colon VL_n \longrightarrow \mathrm{GL}_n(\mathbb{C})$ 
\[
\omega_2(\ell_i)= I_n
 \text{ and } \omega_1(\rho_i)=
\begin{pmatrix}
I_{i-1} & 0 & 0 \\
0 &
\begin{pmatrix}
0 & x \\
\dfrac{1}{x} & 0
\end{pmatrix}
& 0 \\
0 & 0 & I_{n-i-1}
\end{pmatrix} ,
\]
where $x \in \mathbb{C}^*$.
\end{itemize}
Moreover, among the two representations of $VL_n$ above, only $\omega_1$ defines a representation of $WL_n$.
\end{theorem}

\begin{proof}
Since $\omega$ is a non-trivial complex homogeneous $2$-local representation of $VL_n$, we may write
\[
\omega(\ell_i)=
\begin{pmatrix}
I_{i-1} & 0 & 0 \\
0 &
\begin{pmatrix}
a & b \\
c & d
\end{pmatrix}
& 0 \\
0 & 0 & I_{n-i-1}
\end{pmatrix} \text{ and }  \omega(\rho_i)=
\begin{pmatrix}
I_{i-1} & 0 & 0 \\
0 &
\begin{pmatrix}
w & x \\
y & z
\end{pmatrix}
& 0 \\
0 & 0 & I_{n-i-1}
\end{pmatrix}
,
\]
where $a,b,c,d,w,x,y,z \in \mathbb{C}$ satisfy $ad-bc \neq 0$, $wz-xy\neq 0$, and the inner $2\times 2$ blocks are not both equal to the identity matrix $I_2$. Now, from the structure of homogeneous $2$-local representations of $VL_n$, it suffices to consider the following defining relations of $VL_n$, since all other relations yield similar constraints:
\[
\ell_1^2 = 1, \ \ell_1\ell_2\ell_1 = \ell_2\ell_1\ell_2, \ \rho_1^2=1, \ \rho_1\rho_2\rho_1=\rho_2\rho_1\rho_2, \text{ and} \ \rho_1\rho_2\ell_1=\ell_2\rho_1\rho_2.
\]
Applying these relations to the image under $\omega$ produces a system of equations, which can be solved as in Theorem \ref{TH13} to obtain the desired result. Furthermore, when imposing the additional relation $\rho_1 \ell_2 \ell_1 = \ell_2 \ell_1 \rho_2$ of $WL_n$, only $\omega_1$ satisfies it, as required.
\end{proof}

The following theorem gives a complete analysis of the faithfulness of the representation $\omega_1$. We focus solely on this representation, as $\omega_2$ is trivial on one of the generators.

\begin{theorem}\label{Thmfa111}
Consider the representation $\omega_1$ determined in Theorem \ref{Them111}. The following holds.
\begin{itemize}
    \item[(1)] For $n=2$, $\omega_1$ is unfaithful if and only if $\dfrac{b}{x}$ is an $m$-th root of unity for some positive integer $m$.
    \item[(2)] $\omega_1$ is unfaithful for all $n\geq 3$.
\end{itemize}
\end{theorem}

\begin{proof}
We consider each case separately.
\begin{itemize}
    \item[(1)] \underline{The case $n=2$}: The group \(VL_n\) is generated by the two elements \(\ell_1\) and \(\rho_1\), with defining relations
\[
\ell_1^2=\rho_1^2=1.
\]
Consequently, every element of \(VL_n\) can be written as one of the following:
\[
\ell_1,\quad \rho_1,\quad (\ell_1\rho_1)^m,\quad (\rho_1\ell_1)^m,
\qquad m\in\mathbb{Z}.
\]
Clearly, neither \(\ell_1\) nor \(\rho_1\) belongs to \(\ker(\omega_1)\), since their images under \(\omega_1\) are not the identity matrix. Next, we compute the images of the remaining elements. For \(m\in\mathbb{Z}\), we obtain
\[
\omega_1\big((\ell_1\rho_1)^m\big)
=
\begin{pmatrix}
\left(\dfrac{b}{x}\right)^m & 0 \\
0 & \left(\dfrac{x}{b}\right)^m
\end{pmatrix}
\]
and
\[
\omega_1\big((\rho_1\ell_1)^m\big)
=
\begin{pmatrix}
\left(\dfrac{x}{b}\right)^m & 0 \\
0 & \left(\dfrac{b}{x}\right)^m
\end{pmatrix}.
\]
Therefore, the representation \(\omega_1\) is unfaithful if and only if there exists a nonzero integer \(m\in\mathbb{Z}^*\) such that
\[
\left(\dfrac{b}{x}\right)^m = 1.
\]
\item[(2)] \underline{The case $n \geq 3$}: We divide this case into two subcases:
\begin{itemize}
    \item [$\bullet$] If $n=3$, then direct computations imply that $\omega_1((\ell_1\rho_2)^3)=I_3$ with $(\ell_1\rho_2)^3$ a non-trivial element in $VL_3$.
\item[$\bullet$] If $n\geq 4$, we have for all $1 \leq i,j \leq n-1$ with $|i-j| \geq 2$, a direct computation shows that $\omega_1(\ell_i)\omega_1(\ell_j)=\omega_1(\ell_j)\omega_1(\ell_i).$
However, since the commutative relation $\ell_i\ell_j=\ell_j\ell_i$ does not hold in $VL_n$, it follows that the representation $\omega_1$ is unfaithful in this case.\end{itemize}
\end{itemize}
\end{proof}

\begin{theorem} \label{irrtheee}
For $n\geq 3$, the representation $\omega_1$ is irreducible if and only if $b\neq x$.
\end{theorem}

\begin{proof}
 We introduce a representation equivalent to $\omega_1$, denoted by $\omega_1'$, as follows.
Let
\[
P=\mathrm{diag}\big(b^{n-1},\, b^{n-2},\,\ldots,\, b,\,1\big)
\]
be an invertible diagonal matrix. Define
\[
\omega_1'(g)=P^{-1}\omega_1(g)P \qquad \text{for all } g\in VL_n.
\]
Direct computations show that for the generators $\ell_i$ and $\rho_i$, $1\le i\le n-1$, we have
\[
\omega_1'(\ell_i)=
\begin{pmatrix}
I_{i-1} & 0 & 0 \\
0 &
\begin{pmatrix}
0 & 1 \\
1 & 0
\end{pmatrix}
& 0 \\
0 & 0 & I_{n-i-1}
\end{pmatrix} \ \text{ and }\
\omega_1'(\rho_i)=
\begin{pmatrix}
I_{i-1} & 0 & 0 \\
0 &
\begin{pmatrix}
0 & \dfrac{x}{b} \\
\dfrac{b}{x} & 0
\end{pmatrix}
& 0 \\
0 & 0 & I_{n-i-1}
\end{pmatrix},
\]
where $b,x\in\mathbb{C}^*$.

\medskip

\noindent\textbf{Necessary condition.}
Suppose that $b=x$. Then the vector $(1,1,\ldots,1)^T$
is invariant under the action of both $\omega_1'(\ell_i)$ and $\omega_1'(\rho_i)$ for all $1\le i\le n-1$.
Hence $\omega_1'$ admits a nontrivial invariant subspace and is therefore reducible.
Consequently, $\omega_1$ is reducible.

\medskip

\noindent\textbf{Sufficient condition.}
Assume now that $b\neq x$ and suppose, for contradiction, that $\omega_1$ is reducible, then $\omega_1'$ is also reducible. So, there exists a nontrivial subspace $U\subset \mathbb{C}^n$ invariant under $\omega_1'$.
Choose a nonzero vector $u=(u_1,u_2,\ldots,u_n)^T\in U$. For each $1\le i\le n-1$, we have
\[
\omega_1'(\ell_i)u-u=(u_{i+1}-u_i)(e_i-e_{i+1})\in U,
\]
where $e_1,\ldots,e_n$ denotes the standard basis of $\mathbb{C}^n$.
We may choose $u$ so that $u_j\neq u_{j+1}$ for some $1\le j\le n-1$; otherwise $U$ would be spanned by the column vector $(1,1,\ldots,1)^T$, which is impossible since $U$ is invariant under $\omega_1'(\rho_i)$ and $b\neq x$. Thus $e_j-e_{j+1}\in U$ for some $j$.
Applying $\omega_1'(\ell_{j+1})$ and $\omega_1'(\ell_{j-1})$, we obtain
\[
e_{j+1}-e_{j+2}\in U
\quad\text{and}\quad
e_{j-1}-e_j\in U.
\]
Iterating this process yields
\begin{equation}\label{eq:diffs}
e_i-e_{i+1}\in U \qquad \text{for all } 1\le i\le n-1.
\end{equation}
From \eqref{eq:diffs}, no standard basis vector $e_i$ belongs to $U$; otherwise all $e_i$ would lie in $U$, implying $U=\mathbb{C}^n$, a contradiction. Now, consider
\[
\omega_1'(\rho_1)(e_1-e_2)+\frac{x}{b}(e_1-e_2)
=\left(\frac{b}{x}-\frac{x}{b}\right)e_2\in U.
\]
Since $e_2\notin U$, we must have
\[
\frac{b}{x}-\frac{x}{b}=0,
\]
which implies that $b=\pm x$.
Because $b\neq x$, it follows that $b=-x$. In this case, we have
\[
\omega_1'(\rho_i)=
\begin{pmatrix}
I_{i-1} & 0 & 0 \\
0 &
\begin{pmatrix}
0 & -1 \\
-1 & 0
\end{pmatrix}
& 0 \\
0 & 0 & I_{n-i-1}
\end{pmatrix}.
\]
Finally, we compute
\[
\omega_1'(\ell_1)(e_2-e_3)-\omega_1'(\rho_1)(e_2-e_3)
=2e_1\in U,
\]
which implies that $e_1\in U$, contradicting the earlier conclusion that no $e_i$ lies in $U$. Therefore, $\omega_1'$ is irreducible when $b\neq x$, and hence $\omega_1$ is irreducible.

\end{proof}

At the end of this section, a natural question arises.

\begin{question}
Beyond the $2$-local representation of $L_n$, can more general representations of the triplet group $L_n$, like the Tits representation, be extended to $VL_n$ and $WL_n$, and under what conditions would such extensions exist?
\end{question}

\subsection*{Acknowledgement}
The second  author was supported by  United Arab Emirates University under  UPAR grant $\# G00005447$.

\end{document}